\documentclass[12pt,reqno]{amsart}
\usepackage{amssymb}
\usepackage{amsmath}
\usepackage{amsthm}
\usepackage{mathptmx}
\usepackage{color}
\usepackage[pdftex]{graphicx}
\usepackage{hyperref}

\newcommand{\R}{\mathbb R}

\newcommand{\Z}{\mathbb Z}

\newcommand{\hil}{\mathcal{H}}

\newcommand{\para}{\partial^{\alpha}}
\newcommand{\px}{\partial_x}
\newcommand{\pxmu}{\partial_x^{-1}}
\newcommand{\pxd}{\partial_x^2}
\newcommand{\pxt}{\partial_x^3}
\newcommand{\pxc}{\partial_x^4}
\newcommand{\py}{\partial_y}
\newcommand{\pyd}{\partial_y^2}
\newcommand{\pyt}{\partial_y^3}
\newcommand{\pyc}{\partial_y^4}

\newcommand{\p}{\partial}

\newcommand{\al}{{\alpha}}

\newtheorem{theorem}{Theorem}[section]

\newtheorem{remark}[theorem]{Remark}
\newtheorem{lemma}[theorem]{Lemma}
\newtheorem{corollary}[theorem]{Corollary}

\newtheorem*{TA}{Theorem A}
\newtheorem*{TB}{Theorem B}
\newtheorem*{TC}{Theorem C}

\numberwithin{equation}{section}

\begin{document}
\title[Propagation of  regularity of the KPII equation]{On the propagation of regularity of solutions of the Kadomtsev-Petviashvilli (KPII) equation}
\author{Pedro Isaza}
\address[P. Isaza]{Departamento  de Matem\'aticas\\
Universidad Nacional de Colombia\\ A. A. 3840, Medellin\\Colombia}
\email{pisaza@unal.edu.co}
\author{Felipe Linares}
\address[F. Linares]{IMPA\\
Instituto Matem\'atica Pura e Aplicada\\
Estrada Dona Castorina 110\\
22460-320, Rio de Janeiro, RJ\\Brazil}
\email{linares@impa.br}

\author{Gustavo Ponce}
\address[G. Ponce]{Department  of Mathematics\\
University of California\\
Santa Barbara, CA 93106\\
USA.}
\email{ponce@math.ucsb.edu}

\begin{abstract} We shall deduce some special regularity properties of  solutions to the IVP associated to the KPII equation.
Mainly, for datum  $u_0\in X_s(\R^2)$, $s>2$, (see \eqref{XS} below) whose restriction  belongs to $H^m((x_0,\infty)\times\R)$ 
for some $m\in\Z^+,\,m\geq 3,$ and $x_0\in \R$, we  shall prove that the restriction of the corresponding solution $u(\cdot,t)$ 
belongs to $H^m((\beta,\infty)\times\R)$ for any $\beta\in \R$ and any $t>0$. 
\end{abstract}

\maketitle
\section{Introduction}
We consider solutions of the initial value problem (IVP) associated to the  Kadomtsev-Petviashvilli (KPII) equation,
\begin{equation}\label{KP}
\begin{cases}
\partial_tu+\partial^3_xu+\alpha\partial_x^{-1}\partial^2_yu+ u\,\partial_xu=0,\hskip.3cm (x,y)\in \mathbb R^2, \;t>0, \;\alpha=1,\\
u(x,y,0)=u_0(x,y),
\end{cases}
\end{equation}
the operator $\partial_x^{-1}$ is defined via the Fourier transform by
$$
\widehat{\partial_x^{-1}f}(\xi,\eta)=-\frac{i}{\xi}\widehat{f}\,(\xi,\eta).
$$

The KP equations (KPI  ($\alpha=-1$)  and KPII  ($\alpha=1$))  are models for the propagation of long, dispersive,
weakly nonlinear waves which travel predominantly in the $x$ direction, with weak transverse effects. These equations
were derived by Kadomtsev and Petviashvilli \cite{KP-authors} as two-dimensional extensions of the Korteweg-de Vries
equation (see \eqref{kdv} below).  The KP equations have been studied extensively in the last few years in several aspects.  
For an interesting account of KP equations features and open problems we refer the reader to \cite{KS}  (see also \cite{LP}).

Our main purpose in this paper is the study of smoothing properties of solutions of the IVP \eqref{KP}.   
  
Before stating our result we briefly describe the development of the local  well-posedness theory for the IVP \eqref{KP}.  
The first outcome regarding the local  well-posedness of the IVP \eqref{KP}
was given by Ukai in \cite{U} (see also \cite{jcs}, \cite{IMS2}) for initial data in $H^s(\R^2)$, $s\ge3$. In \cite{B}  Bourgain proved local and 
global well-posedness of the IVP \eqref{KP} in $L^2(\mathbb T^2)$ and $L^2(\mathbb R^2)$. Takaoka and Tzvetkov \cite{TT}  and  Isaza and Mej\'{i}a \cite{IM1}  established local well-posedness for data in the anisotropic Sobolev spaces  $H^{s_1,s_2}(\R^2)$, $s_1>-1/3$, $s_2\ge 0$, where
 \begin{equation*}
 H^{s_1,s_2}(\R^2)=\{ f\in \mathcal{S}'(\R^2)\;:\; \|f\|_{H^{s_1,s_2}}=\|\langle \xi\rangle^{s_1}\langle \eta\rangle^{s_2}\widehat{f}\,\|<\infty\},
 \end{equation*}
and $\langle\cdot\rangle^2=1+|\cdot|^2$ (for previous results we refer \cite{Tk1}, \cite{Tz1}, \cite{Tz2}).  Later  Takaoka in \cite{Tk2} proved local well-posedness in $H^{s_1,s_2}(\R^2)$, $s_1>-1/2$, $s_2= 0$, but imposing an additional low frequency condition in the initial data (i.e. $|D_x|^{-\frac12+\epsilon}u_0\in L^2(\R^2)$, for a suitable $\epsilon>0$).  In \cite{H} Hadac removed the latter condition on the initial data and showed local well-posedness for any data in $H^{s_1,s_2}(\R^2)$, $s_1>-1/2$, $s_2\ge 0$. 
Finally, Hadac, Herr and Koch obtained the local well-posedness in the scaling anisotropic Sobolev space $H^{-\frac12, 0}(\R^2)$ for any size data.
They also obtained global well-posedness for small data in the homogeneous anisotropic Sobolev space $\dot{H}^{-\frac12, 0}(\R^2)$ and
local well-posedness in the same space for any size data.  In the anisotropic Sobolev spaces $H^{s_1,s_2}(\R^2)$ the best global result known for any size data was proved by Isaza and Mejia in \cite{IM2} for $s_1>-1/14, \;s_2=0$.  We point out  that the inverse scattering method provides global solution for the KPII equation only  for small initial data (see \cite{W}).
 
In our analysis we will use a result of  Iorio and Nunes \cite{IN} regarding local well-posedness for the KP equations ($\alpha=\pm1$ in \eqref{KP})
and a general nonlinearity $\partial_x F(u)$ in Sobolev spaces $H^s(\R^2)$, $s>2$.  More precisely, we define
\begin{equation}\label{XS}
X_s=\{f\in H^s(\R^2)\;:\; \partial_x^{-1}f\in H^s(\R^2)\}.
\end{equation}

\begin{TA}[\cite{IN}] Let $u_0\in X_s(\R^2)$, $s>2$. There exist $T>0$ and a unique $u=u(x,y,t)$ solution of the IVP
\eqref{KP} such that $u\in C([0,T]; X_s)$. Moreover, the data-solution map is continuous in the $\|\cdot\|_s$--norm.
\end{TA} 

Our main result reads as follows:

\begin{theorem}\label{th.1} For $T>0$, let $u$ be a solution in $[0,T]$ of equation \eqref{KP} with initial data $u_0 \in X_{s}(\mathbb R^2)$, $s>2$. 
Suppose that for an integer $n\geq 3$ and some $x_0\in \R$, the restriction of $u_0$ to $(x_0,\infty)\times \R$ belongs to $H^n((x_0,\infty)\times \R)$
and $\partial_x^{-1}\partial_y^3 u_0\in L^2((x_0,\infty)\times \R)$.

Then, for any $\nu>0$ and $\epsilon>0$
\begin{equation}\label{main}
\sup_{t\in[0,T]}\,\underset{\alpha_1+\alpha_2\le n}{\sum}\,\,\,\int\limits_{-\infty}^{\infty}\,\,\int\limits_{x_0+\epsilon-\nu t}^\infty(\partial_x^{\alpha_1}\partial_y^{\alpha_2} \,u(x,y,t))^2\,dx\,dy<\infty.
\end{equation}
In particular, for all times $t\in (0,T]$ and for all $a\in\mathbb R$, $u(t)\in H^n( (a,\infty)\times\mathbb R)$.
\end{theorem}

\begin{remark} We observe that the condition $\partial_x^{-1}\partial_y^3 u_0\in L^2((x_0,\infty)\times \R)$ is automatically fulfilled if $s\ge 3$.
\end{remark}

\begin {remark} 
From our comments above  and our proof of Theorem \ref{th.1} it will be clear that the requirement $u_0 \in X_s(\R^2)$ in Theorem \ref{th.1} can be
 lowered.
\end{remark}

As a direct consequence of Theorem \ref{th.1} we can deduce


\begin{corollary}\label{cor.1}
Let $\,u\in C(\R : X_s(\R^2))$, $s>2$,  be a solution of the equation in \eqref{KP} described in Theorem A.
 If there exist $ \,m\in\Z^{+},\,m\geq 3,\,\hat t\in\R,\;a\in\R $ such that
\begin{equation*}
u(\cdot,\hat t)\notin H^m((a,\infty)\times \R),
\end{equation*}
then for any $t\in (-\infty,\hat t)$ and any $\beta\in\R$

\begin{equation*}
u(\cdot,t)\notin H^m((\beta,\infty)\times\R).
\end{equation*}
\end{corollary}

 Next, one has  that for appropriate class of data  singularities of the corresponding solutions travel with infinite speed to the left in the $x$-variable as time evolves.

\begin{corollary}\label{cor.2}
Let $\,u\in C(\R: X_s(\R^2))$, $s>2$, be a solution of the equation in \eqref{KP} described in Theorem A.
 If there exist  $ \,k, m \in\Z^{+}$ with  $k\ge m$ and $a,\, b\in \R$ with $\,b<a$ such that
\begin{equation}\label{A}
u_0\in H^k((a,\infty)\times\R)\;\;\;\text{but}\;\;\; u_0\notin H^m((b,\infty)\times\R),
\end{equation}
then for any $t\in(0,\infty)$ and any $v>0$ and $\epsilon>0$
\begin{equation*}
\underset{\alpha_1+\alpha_2\le k}{\sum}\; \:\int\limits_{-\infty}^{\infty}\,\,\,\int\limits_{a+\epsilon-vt}^{\infty} |\partial_x^{\alpha_1}\partial_y^{\alpha_2}\,u(x,y,t)|^2\,dxdy <\infty,
\end{equation*}
and for any $t\in(-\infty,0)$ and  $\gamma\in\R$
\begin{equation*}
\underset{\alpha_1+\alpha_2\le m}{\sum}\; \;\int\limits_{-\infty}^{\infty}\int\limits_{\gamma}^{\infty} |\partial_x^{\alpha_1} \partial_y^{\alpha_2}\,u(x,y,t)|^2\,dxdy =\infty.
\end{equation*}
\end{corollary}

\begin{remark}\hskip10pt

{\rm(a)}  If in Corollary \ref{cor.2} in addition to  \eqref{A} one assumes that
$$
\underset{\alpha_1+\alpha_2\le k}{\sum}\; \;\int\limits_{-\infty}^{\infty}\int\limits_{-\infty}^b |\partial_x^{\alpha_1}\partial_y^{\alpha_2}\,u_0(x,y)|^2\,dxdy<\infty,
$$
then by combining the results in this corollary  with the group properties
it follows that
\begin{equation*}
\underset{\alpha_1+\alpha_2\le m}{\sum}\; \;\int\limits_{-\infty}^{\infty} \int\limits_{-\infty}^{\beta} |\partial_x^{\alpha_1}\partial_y^{\alpha_2}\,u(x,y,t)|^2\,dx =\infty, \hskip10pt \text{for any}\hskip 5pt \beta\in\R\hskip5pt\text{and}\hskip5pt t>0.
\end{equation*}
This shows that the regularity in the left hand side does not propagate forward in time.

{\rm(b)} Notice that \eqref{main}  tells us that the local regularity of the initial datum $u_0$ described in the statement of Theorem \ref{th.1} propagates with infinite speed to its left in the $x$-variable
 as time evolves.

{\rm(c)} In \cite{ILP1} we proved the corresponding result concerning the IVP for the $k$-generalized Korteweg-de Vries equation

\begin{equation}\label{kdv}
\begin{cases}
\partial_t u+\partial_x^3u+u^k\,\partial_x u=0, \hskip10pt x,t\in\R, \;\;k\in \Z^{+},\\
u(x,0)=u_0(x).
\end{cases}
\end{equation}

More precisely,
\begin{TB}\label{propa-kdv}
If  $u_0\in H^{{3/4}^{+}}(\R)$ and for some $\,l\in \Z^{+},\,\;l\geq 1$ and $x_0\in \R$
\begin{equation}\label{notes-3}
\|\,\partial_x^l u_0\|^2_{L^2((x_0,\infty))}=\int_{x_0}^{\infty}|\partial_x^l u_0(x)|^2dx<\infty,
\end{equation}
then the solution of the IVP \eqref{kdv} provided by the local theory satisfies  that for any $v>0$ and $\epsilon>0$
\begin{equation}\label{notes-4}
\underset{0\le t\le T}{\sup}\;\int^{\infty}_{x_0+\epsilon -vt } (\partial_x^j u)^2(x,t)\,dx<c,
\end{equation}
for $j=0,1, \dots, l$ with $c = c(l; \|u_0\|_{{3/4}^{+},2};\|\,\partial_x^l u_0\|_{L^2((x_0,\infty))} ; v; \epsilon; T)$.

In particular, for all $t\in (0,T]$, the restriction of $u(\cdot, t)$ to any interval $(x_1, \infty)$ belongs to $H^l((x_1,\infty))$.

Moreover, for any $v\geq 0$, $\epsilon>0$ and $R>0$ 
\begin{equation}\label{notes-5}
\int_0^T\int_{x_0+\epsilon -vt}^{x_0+R-vt}  (\partial_x^{l+1} u)^2(x,t)\,dx dt< c,
\end{equation}
with  $c = c(l; \|u_0\|_{_{{3/4}^{+},2}};\|\,\partial_x^l u_0\|_{L^2((x_0,\infty))} ; v; \epsilon; R; T)$.
\end{TB}
\end{remark}

\begin{remark}
For solutions of the IVP associated to the Benjamin-Ono equation,
that is,
\begin{equation}\label{BO}
\begin{cases}
\partial_tu-\hil\partial_x^2u+u\partial_x u=0,\hskip.5cm x\in\R, \;t>0,\\
u(x,0)=u_0(x),
\end{cases}
\end{equation}
where $\hil$ denotes the Hilbert transform, we also showed a similar property (see \cite{ILP2}).
\end{remark}

\begin{remark}  In \cite{ILP1} we obtained the following result.
\begin{TC}\label{decay-kdv}
If $u_0\in H^{{3/4}^{+}}(\R)$ and for some $\,n\in \Z^{+},\;n\geq 1$,
\begin{equation}\label{b1}
\|\,x^{n/2} u_0\|^2_{L^2((0,\infty))}=\int_0^{\infty}| \,x^n|\,|u_0(x)|^2dx<\infty,
\end{equation}
then the solution $u$ of the IVP \eqref{kdv} provided by the local theory satisfies
 that
\begin{equation}\label{b2}
\underset{0\le t\le T}{sup}\; \int_0^{\infty} |x^n|\,|u(x,t)|^2\,\,dx\le c
\end{equation}
with $c=c(n;\|u_0\|_{{3/4}^{+},2};\|\,x^{n/2} u_0\|_{L^2((0,\infty))} ; T)$.

Moreover,  for any $\epsilon, \delta, R >0, v\geq 0$, $m,\;j\in\Z^+$, $\;m+j\le n$, $m\ge 1$,
\begin{equation}\label{b3}
\begin{split}
&\underset{\delta \le t\le T}{sup} \;\int_{\epsilon-vt}^{\infty} (\partial_x^mu)^2(x,t)\,x_+^{j}\,dx\\
&\hskip15pt +
\int_{\delta}^T \int_{\epsilon-vt}^{R-vt} (\partial_x^{m+1}u)^2(x,t)\,x_+^{j-1}\,dxdt\le c,
\end{split}
\end{equation}
with $\,c=c(n;\|u_0\|_{{3/4}^{+},2};\|\,x^{n/2} u_0\|_{L^2((0,\infty))} ; T;\delta; \epsilon; R; v)$.
\end{TC}

In \cite{KS} (p.783) Klein and Saut gave
an example showing that initial data in the Schwartz class do not necessarily lead to solutions of the KPII equation in the  Schwartz class.
On the other hand, Levandovsky in \cite{L} showed that for initial data $u_0$ satisfying 
\begin{equation}\label{extra-smoothing}
\int\limits_{\R^2} \big\{u_0^2+(\partial_x^3u_0)^2+(\partial^{-1}_x\partial_y u_0)^2+ x_+^L u_0^2+ x_+^L( \partial_x u_0^2)\big\}\,dxdy<\infty
\end{equation}
for all integer $L\ge 0$,  where $x_{+}=\max\{0,\, x\}$, there exists a unique solution of the IVP \eqref{KP} $u(t)\in C^{\infty}(\R^2)$ for
$t\in(0,T)$.
\end{remark}

We shall notice that solutions of the IVP \eqref{KP} also share a smoothing property similar to the one proved by
Kato (\cite{K}) for solutions of the KdV equation (see \cite{jcs}).

For the generalized KPII equation i.e. 
\begin{equation}\label{gKP}
\partial_tu+\partial^3_xu+\partial_x^{-1}\partial^2_yu+u^p\partial_xu=0\hskip3mm p\in \mathbb Z^{+},\;\;p>1,
\end{equation}
it may be possible to obtain similar results as those in Theorem \ref{th.1}.

This paper is organized as follows. In Section 2, we introduce some tools that will be employed in the proof of Theorem \ref{th.1}.
Section 3 will be devoted to the proof of Theorem \ref{th.1}.

\section{preliminaries}

Our argument of proof uses weighted energy estimates. In this case we will employ weights independent of the variable $y$.
More precisely, for each $\epsilon>0$ and $b\ge 5\epsilon$ we define a function
$\chi_{_{\epsilon, b}} \in C^{\infty}(\R)$ with  $\;\chi_{_{\epsilon, b}}'(x)\ge0$, and
\begin{equation}
\label{2.1}
\chi_{\epsilon,b}(x)=\begin{cases}
0,\hskip10pt x\le \epsilon,\\
1, \hskip 10pt x\ge b,
\end{cases}
\end{equation}
which will be constructed as follows. Let $\rho\in C^{\infty}_0(\R)$, $\rho(x)\geq 0$, even, with $\,\text{supp} \,\rho\subseteq(-1,1)$ and $\,\int\,\rho(x)dx=1$ and define
\begin{equation}\label{2.2}
\nu_{_{\epsilon,b}}(x)=\begin{cases}
0,\hskip75pt x\le 2\epsilon,\\
\\
\frac{1}{b-3\epsilon} x-\frac{2\epsilon}{b-3\epsilon},\hskip15pt x\in[2\epsilon,b-\epsilon],\\
\\
1, \hskip 75pt x\ge b-\epsilon,
\end{cases}
\end{equation}
with
\begin{equation}
\label{2.3}
\chi_{_{ \epsilon, b}}(x)=\rho_{\epsilon}\ast \nu_{\epsilon,b}(x)
\end{equation}
where $\rho_{\epsilon}(x)=\epsilon^{-1}\rho(x/\epsilon)$. Thus
\begin{equation}
\label{2.4}
\begin{split}
& \text{supp}\;\chi_{_{ \epsilon, b}}\subseteq [\epsilon,\infty),\\
& \text{supp}\; \chi_{_{ \epsilon, b}}'(x)\subseteq [\epsilon, b].
\end{split}
\end{equation}
If $\;x\in(3\epsilon, b-2\epsilon)$, then
\begin{equation}
\label{2.5}
\chi_{_{\epsilon, b}}'(x)\geq \frac{1}{b-3\epsilon}.
\end{equation}
and for any $\,x\in \R$
\begin{equation}
\label{2.7}
\chi_{_{\epsilon, b}}'(x) \leq \frac{1}{b-3\epsilon}.
\end{equation}
We will frequently use the following facts
\begin{equation}\label{CL}
\begin{split}
& \chi_{_{\epsilon/5,\epsilon}}(x)=1,\;\;\;\;\;\text{ on supp } \,\chi_{_{\epsilon,b}},\\
&\chi_{_{\epsilon,b}}''(x)   \leq c\,\chi_{_{\epsilon/5,b+\epsilon}}(x).
\end{split}
\end{equation}

Throughout the article we will apply the following inequality of Gagliardo-Nirenberg's type:
\begin{lemma}
 Let $f=f(x,y)$ be a function such that $f\chi\in H^1(\mathbb R^2)$, where $\chi=\chi(x)=\chi_{_{\epsilon, b}}$ is as above. Then,
\begin{equation}\label{GN} 
\Bigl( \int\limits_{\R^2} f^4\chi^2 \Bigr)^{1/2}\leq c\int\limits_{\R^2} f^2\chi+c\int\limits_{\R^2} (\partial_xf)^2\chi+c\int\limits_{\R^2} (\py f)^2\chi
+c\int\limits_{\R^2} f^2\chi'\,.
\end{equation}

\end{lemma}
\begin{proof} It suffices to observe that
\begin{align*}
 f^2(x,y)\chi(x)&\leq\int_{-\infty}^{+\infty}(2|f\partial_xf|\chi+f^2\chi')\,dx\;\quad\text{and}
\\f^2(x,y)\chi(x)&\leq\int_{-\infty}^{+\infty}2|f\partial_yf|\chi\,dy.
\end{align*}
Therefore
\begin{align*}
\iint f^4\chi^2&\leq c\Bigl(\iint(|f\partial_xf|\chi+f^2\chi')\,dx\,dy\Bigr)\Bigl(\iint|f\partial_yf|\chi\,dy\,dx  \Bigr)\,. 
\end{align*}
 In this way, \eqref{GN} follows from Young's inequality.
\end{proof}
\section{proof of Theorem \ref{th.1}}
We begin by giving a brief sketch of the proof. By using a translation in $x$ if necessary we may assume that  $x_0=0$. 
For two integers $\al_1,\al_2$, with $\alpha_1\geq-1$ and $\alpha_2\geq 0$,  let $\al=(\al_1,\al_2)$, $|\al|=\al_1+\al_2$ and $\para=\partial_x^{\al_1}\partial_y^{\al_2}$. We apply $\para$ to  equation \eqref{KP}, multiply by 
$$
 \para u\;\chi\equiv\para u\;\chi_{\epsilon,b}(x+\nu t),
$$
and integrate in $\R^2$. Formally assuming that we have enough regularity to apply integration by parts we obtain that
\begin{equation}\label{ee1}
\begin{split}
\frac12\frac{d\;}{dt}\int(\para u)^2&\chi \,dx\,dy \;\underbrace{-\frac{\nu}2 \int (\para u)^2\chi'\,dx\,dy }_{A_1^\alpha\equiv A_1}
\;\underbrace{ -\frac12 \int(\para u)^2\chi'''\,dx\,dy}_{A_2^\alpha\equiv A_2}\\&
+\frac32\underbrace{\int(\partial_x\para u)^2\chi'\,dx\,dy}_{A_3^\alpha\equiv A_3}+ \frac12\underbrace{\int( \para\partial_x^{-1}\partial_y u)^2\chi' \,dx\,dy}_{A_4^\alpha\equiv A_4}\\&
+\underbrace{\int\para(u\;\partial_xu)\para u\;\chi\,dx\,dy}_{A_5^\alpha\equiv A_5}=0\,.
\end{split}
\end{equation}

In order to write our expressions in a simple form we will use the following notation: for $\alpha=(\alpha_1,\alpha_2)$,
\begin{equation}\label{notation}
\begin{split}
& [\alpha_1,\alpha_2]_{\epsilon,b}\equiv[\alpha_1,\alpha_2]:=\int(\para u)^2\chi_{\epsilon,b}(x+\nu t)\,dx\,dy,\\
&[\alpha_1,\alpha_2]'_{\epsilon,b}\equiv[\alpha_1,\alpha_2]':= A_3^\alpha=\int(\partial_x\para u)^2\chi'_{\epsilon,b}(x+\nu t)\,dx\,dy,\\
&[\alpha_1,\alpha_2]''_{\epsilon,b}\equiv[\alpha_1,\alpha_2]'':= A_4^\alpha=\int( \para\partial_x^{-1}\;\py u)^2\chi'_{\epsilon,b}(x+\nu t)\,dx\,dy.
\end{split}
\end{equation}

When $n\geq 3$, $\chi_{\epsilon,b}(\cdot)u_0\in H^n(\R^2)$ and 
$ \chi_{\epsilon,b}(\cdot)\py^{2}(\px^{-1}\py )u_0\in L^2(\R^2)$,  we will use Gronwall's lemma to show that 
\begin{equation}
\sup_{t\in[0,T]}[\alpha_1,\alpha_2](t)=\sup_{t\in[0,T]}\int(\para u)^2\chi_{\epsilon,b}(x+\nu t)\,dx\,dy\leq C\; \label{gro1}
\end{equation}
 for all indices $\al $ with $3\leq|\al|\leq n$.

By induction we will suppose that \eqref{gro1} is proved for all cases with $|\al|\leq n-1$ and 
 we will refer to a case already proved as a {\it former case}.

For an index $\alpha$ with $|\alpha|=n$,  our procedure will lead to verify that, as a consequence of a former case,
\begin{equation}\begin{aligned}
 \int_0^T|A_1^\alpha(t)|\,dt&\equiv c\int_0^T\int (\para u)^2\chi'\,dx\,dy\,dt\leq C\\
 \text{and }\\
  \int_0^T|A_2^\alpha(t)|\,dt&\equiv c\int_0^T |\int(\para u)^2\chi'''\,dx\,dy|\,dt\leq C.\label{former}
\end{aligned}\end{equation}
Notice that for  $|\alpha|=0,1,2$ with $\alpha_1\geq 0$, inequalities \eqref{gro1} and \eqref{former} follow directly from the well-posedness of the IVP \eqref{KP} with $u_0\equiv u(0)\in H^{2^+}(\R^2)$. Taking into account \eqref{former} and the fact that $A_3^\alpha\geq 0$ and $A_4^\alpha\geq 0$, we will restrict our attention to show that
\begin{equation}
|A_5^\alpha(t)|\equiv\bigl|\int\para(u\;\partial_xu)\para u\;\,\chi\,dx\,dy\bigr|\leq c\int(\para u)^2\chi\,dx\,dy +g(t)\label{V}\,,
\end{equation}
where  $g\geq0 $  is a function with $\int_0^Tg(t)\,dt\leq C$ (sometimes we will mix  several cases together to obtain an inequality similar to \eqref{V}.) We will continue denoting by $g$ a generic nonnegative integrable function on $[0,T]$. 

Once \eqref{V} is obtained,  Gronwall's Lemma  will give \eqref{gro1} for the case $\alpha$ under consideration. Also,  from \eqref{ee1} to \eqref{V} it  will follow that
\begin{equation}\label{smoothing1}
\int_0^T[\alpha_1,\alpha_2]'\,dt\equiv \int_0^TA_3^\alpha(t)\,dt\equiv c\int_0^T\int(\partial_x\para u)^2\chi'\,dx\,dy\,dt\leq C
\end{equation}
 and
 \begin{equation}\label{smoothing2}
\int_0^T\![\alpha_1,\alpha_2]''\,dt\equiv\!\! \int_0^T\! A_4^\alpha(t)\,dt\equiv c\int_0^T\!\!\int(\partial_x^{-1}\partial_y\para u)^2\chi' \,dx\,dy\,dt\leq C\,,
\end{equation}
which guarantees for the case $(\al_1+1,\alpha_2)$ with $\alpha_1\geq -1$ and the case $(\alpha_1-1,\alpha_2+1)$ with $\alpha_1\geq 1$ that
\begin{equation} \label{uno}
\begin{split}
\int_0^T|A_1^{(\alpha_1+1,\alpha_2)}|\,dt&\equiv c\int_0^T\int(\partial_x\para u)^2\chi'\,dx\,dy\,dt\\
&=\int_0^T\!\![\alpha_1,\alpha_2]'\,dt\leq C,
\end{split}
\end{equation}
and
\begin{equation}\label{dos}
\begin{split}\;\int_0^T|A_1^{(\alpha_1-1,\alpha_2+1)}|\,dt&\equiv  c\int_0^T\int(\partial_x^{\alpha_1-1}\partial_y^{\alpha_2+1}u)^2\chi'\,dx\,dy\,dt\\
&=\int_0^T[\alpha_1,\alpha_2]''\,dt\leq C.
\end{split}
\end{equation}
Since $|\chi_{\epsilon,b}'''|\leq c\chi'_{\epsilon/5,b+\epsilon}$, we will  have that
 \begin{equation}
 \int_0^T|A_2^{(\alpha_1+1,\alpha_2)}|\,dt\leq C,\quad\text{and }\;\int_0^T|A_2^{(\alpha_1-1,\alpha_2+1)}|\,dt\leq C.\label{tresu}
\end{equation}
In this way  \eqref{uno}, \eqref{dos}, and \eqref{tresu} will give \eqref{former}  for the cases $(\al_1+1,\alpha_2)$ and $(\alpha_1-1,\alpha_2+1)$.

\vskip4pt
We now begin the proof by considering the cases with $|\alpha|=2$, $\alpha_1\geq 0$. Though the regularity of the solution provides \eqref{gro1} for these cases, we consider them in order to establish the local smoothing effects expressed in \eqref{smoothing1} and \eqref{smoothing2}, which will be used in future cases.
\vspace{5pt}

\noindent\underline{Case (2,0)}:

With $\alpha=(2,0)$, $\partial^\alpha=\partial_x^2$   we estimate the cubic term $A_5$ in \eqref{ee1}. Using integration by parts and Sobolev's embeddings,
\begin{align*}
|A_5|&=|\int\partial_x^2(u\;\partial_xu)\partial_x^2u\;\chi|
=|\int 3\px u(\pxd u)^2\chi+u\;\pxt u\;\pxd u\;\chi|\\
&=| \frac52\int \partial_xu(\partial_x^2u)^2\chi-\frac12 \int u(\partial^2_xu)^2\chi'|
\\
&\leq c(\|\partial_xu\;\|_{L^\infty_{xy}}+\|u\;\|_{L^\infty_{xy}})\|\partial_x^2u\;\|^2_{L^2_{xy}}\leq c\|u\;\|^3_{C([0,T];H^{2^+}(\R^2)}\,  
\end{align*}
Besides,
\begin{equation*}
|A_1|+|A_2|\leq c|\int (\partial_x^2u)^2\chi'|+c|\int(\partial_x^2u)^2\chi'''|\leq c\|u\;\|^2_{C([0,T];H^{2^+}(\R^2)}.
\end{equation*}
 Thus, by integrating  \eqref{ee1} in $[0,T]$, and taking into account that the values of $[2,0]$ at $t=0$ and at $t=T$  are bounded by $c\|u\;\|^2_{C([0,T];H^{2^+}(\R^2)}$, we obtain \eqref{smoothing1}  and \eqref{smoothing2} for the case $(2,0)$, which, according to our notation \eqref{notation},  is
 \begin{equation}
 \int_0^T([2,0]'+[2,0]'')\,dt\leq C.\label{2}
 \end{equation}
 Notice that this estimate provides  \eqref{former} for the future case $\alpha=(3,0)$.
 
 \vspace{5pt}
 \noindent\underline{Case (1,1)}:
 
 With $\alpha=(1,1)$ and $\para=\partial_x\partial_y$, we apply integration by parts to obtain  that
 \begin{equation}
 \begin{split}
 |A_5|&=|\int (2\px\py  u\;\px u+\py u\;\pxd u+u\;\py\pxd u)\px\py u\;\chi|\\
 &=|\int \textstyle{\frac32}\px u(\px\py u)^2\chi-\textstyle{\frac 12} u(\px\py u)^2\chi'+\py u\;\pxd u\;\px\py u\;\chi|\\
 &\leq \|u\;\|^3_{C([0,T];H^{2^+}(\R^2)},
 \end{split}
 \end{equation} 
 and, proceeding as in the former case we have that
  \begin{equation}\int_0^T([1,1]'+[1,1]'')\,dt
  \label{uuno}\leq C,
  \end{equation}
  which gives \eqref{former} for the case $\alpha=(2,1)$.
  
 \vspace{5pt}
 \noindent \underline{Case (0,2)}:
 
 The cubic term $A_5$ with $\alpha=(0,2)$ in  \eqref{ee1}, is treated as the former cases to obtain that
 \begin{equation*}
 |A_5|=|\int \textstyle{\frac12}\px u(\pyd u)^2\chi-\textstyle{\frac12} u\;\,(\pyd u)^2\chi' +2\py u\;\,\px\py u\;\,\pyd u\;\,\chi|\leq C,
 \end{equation*}
 and from this estimate we then have that
  \begin{equation}
 \int_0^T([0,2]'+[0,2]'')\,dt=\int_0^T\int(\px\pyd u)^2\chi'\,dx\,dy\,dt\leq C,\label{ddos}
 \end{equation}
 to be used in the case $\alpha=(1,2)$.
 
 For  the estimations of order $|\alpha|= 3$ we will  need to consider a single case with $\alpha_1=-1$, namely the case (-1,3). \vskip6pt
 
 \vspace{5pt}
 \noindent \underline{Case (-1,3):}
 
For this case  $\partial^{\alpha}=\px^{-1}\pyt$. From integration by parts and Young's inequality it follows that
 \begin{align*}
 |A_5|&=\frac12|\int \pyt u^2\px^{-1}\pyt u\;\,\chi|=\frac12|\int (2u\;\pyt u+6\py u\;\py^2 u)\p_x^{-1}\pyt u\;\,\chi|\\
 &\leq\frac12|-\int\px u(\pxmu \pyt u)^2\chi-\int u(\pxmu \pyt u)^2\chi'|\\&\quad\quad+c\|\partial_yu\;\|_{L^\infty} \int |\pyd u||\pxmu\;\pyt u|\chi\\
 &\leq c\|\partial_xu\;\|_{L^\infty}[-1,3]+c\|u\;\|_{L^\infty}[0,2]''+c[0,2]+c[-1,3]\,\\
 &\leq c+c[-1,3]+g(t)\,.
 \end{align*}
   On the other hand, since $|\chi'''|\leq c \chi'_{\epsilon/5,b+\epsilon}$, we see that in this case
  \begin{equation*}
  |A_1+A_2|\leq c\int(\px^{-1}\pyd u)^3\chi'_{\epsilon/5,b+\epsilon}\,dt\leq c[0,2]''_{\epsilon/5,b+\epsilon}\,.
  \end{equation*}
  In this way, from the above estimates  
  \begin{equation*}\frac{d}{dt}[-1,3]\leq c[-1,3]+g(t)\,,  \end{equation*}
  which gives \eqref{gro1}, \eqref{smoothing1}, and \eqref{smoothing2} for this case.
 \vskip 5pt
 We now turn to the cases with $|\alpha|=3$ . Thus we assume that $u_0$ satisfies the hypotheses in the statement of Theorem \ref{th.1} with $n=3$.

  \vspace{5pt}
 \noindent\underline{Case (3,0)}:
 
 From  integration by parts we see that
 \begin{align*}
 A_5&=\int \pxt u(u\px u)\pxt u\;\,\chi=\int(4\px u\;\pxt u+3\pxd u\;\pxd u+u \;\px^4 u)\,\pxt u\;\,\chi\\
 &={\textstyle{\frac72}}\int\px u(\pxt u)^2\,\chi-{\textstyle{\frac12}}\int u (\pxt u)^2\,\chi'{\textstyle{-{\textstyle{3}}\cdot\frac13}}\int (\pxd u)^3\chi'
 \equiv A_{51}+A_{52}+A_{53}\,.
  \end{align*}
  By Sobolev embeddings
  \begin{equation}\label{4}
  \begin{split}
  |A_{51}|+|A_{52}|&\leq (\|\partial_x u\;\|_{L^\infty}+\|u\;\|_{L^\infty})\int (\pxt u)^2\,\chi+ c\int(\pxt u)^2\,\chi'\\
  &\leq c[3,0]+c[2,0]'.
  \end{split}
  \end{equation}
 The first term on the right hand side of \eqref{4} is the quantity to be estimated while the second term has finite integral in $[0,T]$ by 
 \eqref{2}.
 Now, from integration by parts and Young's inequality
 \begin{equation}\label{6}
 \begin{split} 
 |A_{53}|&=|-2\int \px u\;\,\pxd u\;\,\pxt u\;\,\chi'-\int\px u(\pxd u)^2\chi''|\\
 &\leq
 c\|\partial_x u\;\|_{L^\infty}\bigl(\int(\pxd u)^2\chi'+\int (\pxt u)^2\,\chi' +\int (\pxd u)^2|\chi''|)\\
 &\leq c+c[2,0]'+c,
 \end{split}
 \end{equation}
 which is bounded after integration in $[0,T]$.
 
 Since from the case (2,0), and inequalities \eqref{uno} and \eqref{tresu}  we have that $|A_1|$ and $|A_2|$ have finite integral in $[0,T]$, it follows that \begin{equation*}
 \frac{d}{dt}[3,0]\leq c[3,0]+g(t).
 \end{equation*}
 Therefore, as we have shown in the sketch of our proof, we obtain \eqref{gro1} \eqref{smoothing1}, and \eqref{smoothing2} for the case $(3,0)$. That is 
 \begin{equation*}\sup_{t\in[0,T]}[3,0]\leq C \quad\text{and }\int[3,0]'+[3,0]''\,dt<\infty.\end{equation*}
 
 We will now turn to the cases (2,1), (1,2), and (0,3). As it will be seen,    we need to consider these three cases  together for the application of Gronwall's lemma.

  \vspace{5pt}
 \noindent\underline{Case (2,1):}
 
 We have that
 \begin{align*}|A_5|&=\int\pxd\py (u\px u)\chi\\&=\int(a_1\pxd\py u \,\px u+a_2\px\py u\;\pxd u+a_3\py u\;\,\pxt u+u\;\pxt \py u)\pxd\py u\;\chi\\&\equiv A_{51}+A_{52}+A_{53}+A_{54}.
 \end{align*}
 We apply  Young's inequality and Sobolev embeddings to obtain that
 \begin{equation}\label{7}
 \begin{split}
 |A_{51}+A_{53}|&\leq c\|\partial_xu\;\|_{L^\infty}[2,1]+ c\|\py u\;\|_{L^\infty}([3,0]+[2,1])\\
 &\leq c+c[2,1],
 \end{split}
 \end{equation}
 since (3,0) is a former case and we have already seen that $[3,0]\leq c$.
 
 From integration by parts it follows that
 \begin{equation}\label{7-1}
 \begin{split}
 |A_{54}|&=|-\frac12\int \px u(\pxd\py u)^2\chi-\frac12 \int  u(\pxd\py u)^2\chi'|\\
 &\leq c[2,1]+c[1,1]'.
 \end{split}
 \end{equation}
 
 For  $A_{52}$,  we integrate by parts to conclude that
  \begin{equation*}
 A_{52}=-\frac{a_2}2\int (\px\py u)^2\pxt u\;\chi-\frac{a_2}2\int (\px\py u)^2\pxd u\;\chi' \equiv A_{521}+A_{522}.
 \end{equation*}
 To estimate $A_{521}$ we apply \eqref{GN} and the facts that $\chi_{\epsilon,b}=\chi_{\epsilon,b}\chi_{\epsilon/5,\epsilon}$ and $\chi_{\epsilon/5,\epsilon}^2\leq \chi_{\epsilon/5,\epsilon}$ to conclude that
 \begin{equation}\label{7-2}
 \begin{split}
& |A_{521}|=c| \int (\px\py u)^2\pxt u\;\chi\chi_{\epsilon/5,\epsilon}| \\
 &\leq \Bigl(\int (\px\py u)^4\chi^2\Bigr)^{1/2}\Bigl(\int\!(\pxt u)^2 \chi_{\epsilon/5,\epsilon}^2 \Bigr)^{1/2}\\
 &\leq c[3,0]^{1/2}_{\epsilon/5,\epsilon}\bigl(\!\int(\px\py u)^2\chi\!+\!(\pxd\py u)^2\chi\!+\!(\px\pyd u)^2\chi\!+\!(\px\py u)^2\chi'\bigr)\\
 &\leq c([1,1]+[2,1]+[1,2]+[0,1]')\\
 &\leq c+c[2,1]+c[1,2]+c[0,1]',
 \end{split}
 \end{equation}
 since the cases (3,0) and (1,1) are former cases.  
 
 $A_{522}$ can be treated in a similar manner to obtain that
 \begin{equation}\label{8}
 \begin{split}
 |A_{522}|&\leq \Bigl(  \int(\px\py u)^4(\chi')^2 \Bigr)^{1/2}\Bigl( \int(\pxd u)^2\chi_{\epsilon/5,\epsilon}^2   \Bigr)^{1/2}\\
 &\leq c[2,0]_{\epsilon/5,\epsilon}^{1/2}\bigl( \int (\px\py u)^2\chi'+(\pxd\py u)^2\chi'\\
 &\hskip10pt +(\px\pyd u)^2\chi'+(\px\py u)^2|\chi''|\bigr)\\
 &\leq c([0,1]'+[1,1]'+[0,2]'+[0,1]'_{\epsilon/5, b+\epsilon}),
 \end{split}
 \end{equation}
 since 
 $$
 |\chi''|= |\chi''_{\epsilon,b}|\leq c\,\chi'_{\epsilon/5, b+\epsilon}
 $$ 
 and (2,0) is a former case. 
 
 On the other hand,
 \begin{equation*}
 \begin{split}
 |A_1|+|A_2| &\leq c|\int(\pxd\py u)^2\chi'|+c|\int(\pxd\py u)^2\chi'''|\\
 &\leq c[1,1]'+[1,1]'_{\epsilon/5,b+\epsilon}\,.
 \end{split}
 \end{equation*}
 In this way, gathering the above estimates, and taking into account that the cases (0,1), (0,2), and (1,1)  are former cases we conclude that
 \begin{equation}\label{10}
\frac{d}{dt}[2,1]\leq c[2,1]+c[1,2]+g(t).
\end{equation}

 \vspace{5pt}
 \noindent\underline{Case (1,2):}
\begin{align*}
&|A_5|=\int\px\pyd(u\px u)\px\pyd u\;\,\chi
\\&=\int(a_1\px\pyd u\;\,\px u+a_2\px\py u\;\px\py u+a_3\pyd u\;\pxd u\\&\quad\quad+a_4\py u\;\pxd\py u+u\;\pxd\pyd u)\px\pyd u\;\chi
\\&\equiv A_{51}+A_{52}+A_{53}+A_{54}+A_{55}.
\end{align*}
Integrating by parts in the term $A_{55}$, and proceeding as we did to obtain  \eqref{7} and \eqref{7-1},  we have that 
\begin{equation*}
|A_{51}+A_{54}+A_{55}|\leq  c[1,2]+c([1,2]+[2,1])+ c([1,2]+c[0,2]')\,.
\end{equation*}
Integration by  parts with respect to $y$ shows that $A_{52}=0$. 

For $A_{53}$ we integrate by parts and apply  \eqref{GN} to conclude that
\begin{align*}
&|A_{53}|=\Bigl|-\frac{a_3}2\int (\pyd u)^2\pxt u\;\chi-\frac{a_3}2\int (\pyd u)^2\pxd u\;\chi'\Bigr|\\
&= \Bigl|-\frac{a_3}2\int (\pyd u)^2\;\chi\;\pxt u\;\chi_{\epsilon/5,\epsilon}+\frac {a_3}2\Bigl(\int 2\pyd u\;\px\pyd u\;\px u\,\chi'+\int (\pyd u)^2\px u\;\chi''\Bigr)\Bigr|\\
&\leq c [3,0]_{\epsilon/5,\epsilon}^{1/2}\Bigl(\int(\partial_y^2 u)^4\chi^2   \Bigr)^{1/2} +c\int |\pyd u\;\px\pyd u|\;\chi'+c\int (\pyd u)^2\;\chi'_{\epsilon/5,b+\epsilon}\\
&\leq c([0,2]+[1,2]+[0,3]+[1,1]'')+c([1,1]'' +[0,2]') +c[1,1]''_{\epsilon/5,b+\epsilon}\,.\\
\end{align*}

Also,
\begin{equation*}
|A_1|+|A_2|\leq c\int(\px\pyd u)^2\chi'_{\epsilon/5,b+\epsilon}\leq c[0,2]'.
\end{equation*}
From the above estimates and taking into account that $(0,2)$ and $(1,1)$ are former cases we have that
\begin{equation}\label{12}
\frac{d}{dt}[1,2]\leq c[2,1]+c[1,2]+c[0,3]+g(t).
\end{equation}

 \vspace{5pt}
 \noindent\underline{Case (0,3)}:

\begin{align*}
A_5&=\int \pyt(u\px u)\pyt u\;\chi\\&=\int (\pyt u\px u+3\pyd u\;\px\py u+3\py u\;\px\pyd u+u\;\px\pyt u)\pyt u\;\chi\\
&\equiv A_{51}+A_{52}+A_{53}+A_{54}.
\end{align*}
From Sobolev embeddings and Young's inequality
\begin{equation*}
|A_{51}+A_{53}|\leq c[0,3]+c([1,2]+[0,3]).
\end{equation*}
Applying integration by parts we obtain
\begin{align*}
|A_{52}|&=|\frac32\int\px\pyd u(\pyd u)^2\chi|=|\frac12\int(\pyd u)^3\chi'|\\
&=|\int \py u\;\pyt u\;\pyd u\;\,\chi'|\leq c\int(\pyt u)^2\chi'+c(\pyd u)^2\chi'\\
&\leq c[-1,3]'+c[1,1]''.
\end{align*}
For $A_{54}$ we see that
\begin{align*}
|A_{54}|&=|-\frac12\int\px u(\pyt u)^2\chi-\frac12\int u(\pyt u)^2\chi'|\\
&\leq c[0,3]+c[-1,3]'.
\end{align*}
Also 
\begin{align*}
|A_1|+|A_2|\leq c\int (\pyt u)^2\chi'_{\epsilon/5,b+\epsilon}\leq c[-1,3]'_{\epsilon/5,b+\epsilon}\,.
\end{align*}
In this way  we see that
\begin{equation}
\frac{d}{dt}[0,3]\leq c[0,3]+c[1,2])+g(t),\label{14}
 \end{equation}
 since the cases $(-1,3)$ and $(1,1)$ are former cases.
 
 Hence, from \eqref{10}, \eqref{12}, and \eqref{14} it follows that
 \begin{equation*}\frac{d}{dt}([2,1]+[1,2]+[0,3])\leq c([2,1]+[1,2]+[0,3])+g(t),  \end{equation*}
 which gives \eqref{gro1}, \eqref{smoothing1}, and \eqref{smoothing2} for the three cases $(2,1)$, $(1,2)$, and $(0,3)$ together. 
 
 For the cases with $|\alpha|=4$ we will see that the case $(4,0)$ can be obtained independently of the other cases of the same order.

 \vspace{5pt}
 \noindent \underline{Case (4,0)}:  
 For this case
 \begin{align*}
 |A_5|&=\bigl|\int \pxc u(u\px u)\,\pxc u \chi=\int  (5\pxc u\;\,\px u+10\pxd u\;\,\pxt u+u\;\px^5 u)\pxc u\;\,\chi  \bigr|\\
 &=\bigl|  {\textstyle -\frac52}\int\px u(\pxc u)^2\,\chi-{\textstyle \frac12}\int u(\pxc u)^2\,\chi'+{10}\int\pxd u\;\pxt u\;\pxc u \,\chi\bigr|
\\&
\leq c[4,0]+c[3,0]'+\bigl| \int(\pxd u)^2(\pxt u)^2\chi\bigr| +c[4,0].
 \end{align*}
 
To estimate the last integral term we will  use the notation $\widetilde{\chi}:=\chi_{\epsilon/5,\epsilon}$ and take into account that   ${\chi}_{\epsilon',b'}^2\leq c\chi_{\epsilon',b'}$ and $(\widetilde{\chi}')^2_{\epsilon',b'}\leq c\widetilde{\chi}'_{\epsilon',b'}$ for $0<\epsilon'<b'/5$. We will aslo apply the following Gagliardo-Nirenberg's inequality:
 \begin{equation*}\int f^6\,dx\,dy\leq c\int f^2\,dx\,dy \bigl(\int \bigl((\partial_x f)^2+(\partial_yf)^2\bigr)\,dx\,dy\bigr)^2\,.\end{equation*}
 In this way,
 \begin{align*}& \bigl|\int(\pxd u)^2(\pxt u)^2\chi\bigr|=|-{\textstyle \frac13} \int(\pxd u)^3\pxc u\;\chi-{\textstyle \frac13}\int(\pxd u)^3\pxt u\;\chi'\bigr|
\\&= \frac13\bigl|\int (\pxd u)^3\widetilde{\chi}^3\pxc u\;\chi+\int(\pxd u)^3\widetilde{\chi}^3\pxt u \chi'\bigr|
\\&
\leq c\int(\pxd u\;\widetilde{\chi})^6+c[4,0]+c\int(\pxd u\;\widetilde{\chi})^6 +c[2,0]'\\
&\leq c\Bigl(  \int(\pxd u)^2 \widetilde{\chi}^2  \Bigr)  \Bigl( \int(\pxt u)^2\widetilde{\chi}^2 +(\pxd u)^2(\widetilde{\chi}')^2+(\py\pxd u)^2
\widetilde{\chi}^2\Bigr)^2\,\\&\quad\quad+c[4,0]+c[2,0]'\\ &
\leq c\|u(t)\|_{H^2}([3,0]^2_{\epsilon/5,\epsilon}+ \|u(t)\|^4_{H^2}+[2,1]^2_{\epsilon/5,\epsilon})+c[4,0]+c[2,0]'
\\
&\leq c+  c[4,0]+ c[2,0]'\,\leq c[4,0]+g(t),
 \end{align*}
which together with \eqref{uno} and \eqref{tresu} gives \eqref{gro1} for this case.
  
 We will now  consider the cases (3,1), (2,2), (1,3) and estimate them together.
 
  \vspace{5pt}
 \noindent\underline{Case (3,1):}
 \begin{align*}
 &A_5=\int \pxt\py(u\px u)\pxt\py u\;\chi\\
 &=\int (a_1\pxt\py u\;\,\px u+ a_2\pxd\py u\;\pxd u+a_3  \pxt u\;\px\py u)\pxt\py u\;\chi\\
&\hskip10pt  +\int (a_4\py u\;\pxc u+u\;\px\pxt\py u)\pxt\py u\;\chi\\
 &=A_{51}+A_{52}+A_{53}+A_{54}+A_{55}\,.
 \end{align*}
 Treating the last term in the former integral by integration by parts and proceeding as we did to obtain \eqref{7} and \eqref{7-1} we have that
 \begin{equation*}
 |A_{51}+A_{54}+A_{55}|\leq c[3,1]+c[4,0]+c[2,1]'\leq c+ c[3,1]+c[2,1]'\,.
 \end{equation*}
 For the remaining terms $A_{52}$ and $A_{53}$ we can use inequality \eqref{GN} to obtain that
 \begin{equation}\label{16}
 \begin{split}
&|A_{52}+A_{53}| \leq c\Bigl( \int(\pxd\py u)^4\chi^2  \Bigr)^{1/2}\Bigl( (\pxd u)^4\widetilde{\chi}^2  \Bigr)^{1/2}\\
 &\hskip7pt +c\Bigl((\pxt u)^4\chi^2   \Bigr)^{1/2}\Bigl((\px\py u)^4  \widetilde{\chi}^2 \Bigr)^{1/2}+c[3,1]\\
 &\leq c([2,1]\!+\![3,1]\!+\![2,2]\!+\![1,1]')([2,0]\!+\![3,0]\!+\![2,1]\!+\![1,0]')_{\epsilon/5,\epsilon}\\
 &\hskip7pt+ c([3,0]\!+\![4,0]\!+\![3,1]\!+\![2,0]')([1,1]\!+\![2,1]\!+\![1,2]\!+\![0,1]')_{\epsilon/5,\epsilon}\\
 &\hskip7pt+c[3,1],
 \end{split}
 \end{equation}
 where the subindex $(\epsilon/5,\epsilon)$ in the closing parenthesis means that all terms $[\cdot,\cdot]$  inside the parentheses are to be taken as 
 $[\cdot,\cdot]_{\epsilon/5,\epsilon}$. 
 Since 
 \begin{equation*}([2,0] +[1,0]'+[1,1]+[0,1]')_{\epsilon/5,\epsilon}\leq c\|u\|_{C([0,T];H^2(\mathbb R^2))}
 \end{equation*}
  and noticing the former cases in \eqref{16}, we conclude that
  \begin{align*}
  |A_{52}+A_{53}|&\leq c+c(  [3,1]+[2,2]+[1,1]'+[4,0]+[2,0]')\\& \leq  c+c(  [3,1]+[2,2]+[1,1]'+[2,0]'),
  \end{align*}
  and therefore, from the above estimates for this case,
  \begin{equation}
  |A_5|\leq c[3,1]+c[2,2]+g(t).\label{17}
  \end{equation}
  
 \vspace{5pt}
 \noindent\underline{Case (2,2):}
\vskip4pt
We proceed as in  case (3,1) to obtain  analogous terms $A_{51}$ to $A_{56}$.  We observe as before that
\begin{equation*}
 |A_{51}+A_{55}+A_{56}|\leq c[2,2]+c[3,1]+c[1,2]',
  \end{equation*}
 while for $A_{52}$, $A_{53}$ and $A_{54}$ we see that
 \begin{equation*}
 A_{52}+A_{53}+A_{54}= \int(a_2\pxd\py u\;\px\py u+a_3\px\pyd u\;\pxd u 
 +a_4\pyd u\;\pxt u)\pxd\pyd u\;\chi,
 \end{equation*}
which can be treated by using inequality \eqref{GN}, as we did in  the case (3,1),       to conclude that
\begin{align*}
|A_{52}&+A_{53}+A_{54}|\leq c[2,2]\\
&+
 c([2,1]+[3,1]+[2,2]+[1,1]')([1,1]+[2,1]+[1,2]+[0,1]')_{\epsilon/5,\epsilon}\\
 &+c([1,2]+[2,2]+[1,3]+[0,2]')([2,0]+[3,0]+[2,1]+[1,0]')_{\epsilon/5,\epsilon}
 \\
 &
 +([3,0]+[4,0]+[3,1]+[2,0]')([0,2]+[1,2]+[0,3]+[1,1]'')_{\epsilon/5,\epsilon}\\
 &\leq c+c([3,1]+[2,2]+[1,1]'+[1,3]+[0,2]'+[4,0]+[2,0]').
  \end{align*}

Thus, for the case (2,2),
\begin{equation}
|A_5|\leq c[3,1]+c[2,2]+c[1,3]+g(t)\,.\label {17a}
\end{equation}
 \vspace{5pt}
 \noindent\underline{Case (1,3)}:

We see that
\begin{align*}
A_5=&\int \px\pyt(u\px u)\px\pyt u\\
=&\int(a_1\px\pyt u\;\px u+a_2\pyt u\;\pxd u+a_3\px\pyd u\;\px\py u+a_4\pyd u\;\pxd\py u\\
&\quad\quad+a_5\py u\;\pxd \pyd u+\,u\;\pxd\pyt u)\chi\;\px\pyt u\\
&=A_{51}+\cdots+A_{56}.\\
\end{align*}
As before,
\begin{equation}
|A_{51}+A_{55}+A_{56}|\leq c[1,3]+c[2,2]+c[0,3]'\,.\label{17b}
\end{equation}
The terms $A_{53}$ and $A_{54}$ can be bounded using inequality \eqref{GN} as it was done to obtain \eqref{16} above. In this case we have
\begin{equation*}
\begin{split}
&|A_{53}+A_{54}|\leq c[1,3]\\
&\hskip10pt+c([1,2]+[2,2]+[1,3]+[0,2]')([1,0] +[2,0] +[1,2] +[0,1]' )_{\epsilon/5,\epsilon}\\
&\leq c([0,2]+[1,2]+[0,3]+[1,1]'')_{\epsilon/5,\epsilon}([2,1]+[3,1]+[2,2]+[1,1]').\\
 \end{split}
 \end{equation*}
 Since 
 $$
 ([0,1]'+[1,1]'')_{\epsilon/5,\epsilon}\leq c\|u\|^2_{C([0,T]; \,H^2(\R^2))},
 $$
 it follows that
 \begin{equation}
 |A_{53}+A_{54}|\leq c+ c([2,2]+[1,3]+ [0,2]'+ [3,1]+[1,1]').\label{18--}
 \end{equation}
 We now consider the term $A_{52}$. For this term we will use the embedding $W^{2,1}(\mathbb R^2)\hookrightarrow L^\infty(\mathbb R^2)$, (where $W^{2,1}$ is the classical Sobolev space of $L^1$ functions having derivatives up to second order in $L^1$). More precisely we will use the inequality
 \begin{equation}
 \|f\|_{L^\infty(\mathbb R^2)}\leq c\|\py\px f\|_{L^1(\mathbb R^2)}.\label{18-}
 \end{equation}
 Now,
 \begin{equation}\label{18}
 \begin{split}
 |A_{52}|&=c|\int \pyt u\;\pxd u\;\px\pyt u\;\chi|\leq c\int (\pyt u)^2(\pxd u)^2\;\chi\;\widetilde{\chi}+c[1,3]\\
 &\leq \|(\pxd u)^2\chi\|_{L^\infty_{x,y}}[0,3]_{\epsilon/5,\epsilon}.
 \end{split}
 \end{equation}
 We estimate the $L^\infty$ norm in \eqref{18} by using inequality \eqref{18-} to conclude that
 \begin{equation}\label{20}
 \begin{split}
 \|(\pxd u)^2&\chi\|_{L^\infty_{x,y}}\leq c\int|\px\py[(\pxd u)^2\chi]|\\
 &= c\int|2\pxd\py u\;\pxt u\;\chi+2\pxd u\;\pxt\py u\;\chi+2\pxd u\;\pxd\py u\;\chi'|\\
 &\leq c[2,1]+[3,0]+c[2,0]+c[3,1]+c[1,0]'+c[1,1]'\,\\
 &\leq c[3,1]+g(t).
 \end{split}
 \end{equation} 
 Therefore, from \eqref{17b}, \eqref{18--}, \eqref{18}, and \eqref{20} it follows that
 \begin{equation}
 |A_5|\leq c[3,1]+[2,2]+[1,3]+g(t). \label{22}
 \end{equation}
From the estimates obtained in \eqref{17}, \eqref{17a}, and \eqref{22} for the cases (3,1), (2,2), and (3,3), we conclude that
\begin{equation*}
\frac d{dt}([3,1]+[2,2]+[1,3])\leq c([3,1]+[2,2]+[1,3])+g(t),
\end{equation*} 
and thus we have \eqref{gro1} for the cases (3,1), (2,2), and (1,3).

Notice that we also have the option of treating the case (4,0)  by using inequalities \eqref{GN} or \eqref{18-} to obtain an estimate for the four terms $[4,0]$, $[3,1]$, $[2,2]$, and $[1,3]$ together.

 \vspace{5pt}
 \noindent\underline{Case (0,4):}

In this case
\begin{align*}
A_{5}=\int \pyc(u\px u)\pyc u\;\chi&=\int (\pyc u\;\px u+a_2 \pyt u\;\px\py u+a_3 \pyd u\;\px\pyd u\\
&\hskip10pt+a_4 \py u\;\px\pyt u+u\;\px\pyc u)\pyc u\;\chi\\&
=A_{51}+\dots+A_{55}.
\end{align*}
As usual, after applying integration by parts in $A_{55}$  we find that
\begin{equation*}
|A_{51}+A_{54}+A_{55}|\leq c[0,4]+c[1,3]+c|\int u(\pyc u)^2\,\chi'|\,.
\end{equation*}
Notice that the last term in the former expression is bounded by $c[1,3]'$ which is also a case of order 4. However,  we took precautions to avoid the appearance of the term $[0,4]$ in our bounds for the other cases of order 4. In particular, we achieve that in the case $(1,3)$ with the application of the embedding $W^{2,1}\hookrightarrow L^\infty$  in \eqref{18} and \eqref{20}.

In this way, since the case $(1,3)$ is already a former case, we conclude that
\begin{equation*}
|A_{51}+A_{54}+A_{55}|\leq c[0,4]+g(t).
\end{equation*}
For the terms $A_{52}$ and $A_{53}$ we observe that they have the form $\int\partial^{\gamma}u\;\partial^\beta u\;\pyc u\;\chi$ with $|\gamma|=3$ and $\beta=2$
allowing us to apply inequality \eqref{GN} as we did in the cases above. Thus, taking into account all former cases, we find that
\begin{equation*}
|A_{52}+A_{53}|\leq c[0,4]+g(t),
\end{equation*}
and we can conclude that \eqref{gro1} is valid for this case.

 \vspace{5pt}
 \noindent\underline{Cases with $n\geq 5$:}

The cases with $|\alpha|=n\geq 5$ are easier since, as we saw for $n=4$, we have enough regularity to estimate the terms with second order derivatives by means of the Sobolev embedding \eqref{18-}. 

We group the cases $(n,0),\,(n-1,1),\cdots,(1,n-1)$, estimate the corresponding integrals $[n,0],\,[n-1,1],\cdots,[1,n-1]$ together in a single application of Gronwall's lemma, and then consider the estimation of  $[0,n]$ separately.

If $\alpha=(\alpha_1,\alpha_2)$, with $\alpha_1+\alpha_2=n$ and $\alpha_1\geq 1$, we observe that the expansion of $A_5=\int\para(u\px u)\para u\chi$ consists of a sum of terms of the form
\begin{equation}
A_\beta=\int\partial^\beta u\;\px\partial^\gamma u\;\chi,\quad \text{with } |\beta|+|\gamma|=n.
\end{equation}
Proceeding as in the cases with $|\alpha|=3,4$ we first consider the terms with $|\beta|=1$ or $|\gamma|=0$,  which can be treated by the Sobolev inequality $\|\partial_x u\|_{L^\infty}+|\partial_y u\|_{L^\infty}\leq c$ and Young's inequality. For the term with $|\beta|=0$ ($|\gamma|=n$), $A_{(0,0)}$, we apply as before integration by parts to obtain the bound
\begin{equation}
|A_{(0,0)}|\leq c \|\partial_x u\|_{L^\infty}[\alpha_1,\alpha_2]+c\|u\|_{L^\infty}[\alpha_1-1,\alpha_2]'.\label{24}
\end{equation}
The terms with $|\beta|=2$ ($|\gamma|=n-2$) or with $|\gamma|=1$ ($|\beta|=n-1$) can now be bounded, as we did in the case (1,3), by using the Sobolev embedding \eqref{18-}. Notice that in the estimates of these cases the term $[0,n]$ will not appear.

The intermediate terms with other combinations of $\beta$ and $\gamma$ will have $|\beta|\leq n-2$ and $|\gamma|\leq n-2$ and can be estimated by means of inequality \eqref{GN} to give bounds which always come from former cases. 
In this way, adding all cases under consideration we have that
\begin{equation*}
\frac{d}{dt}([n,0]+[n-1,0]+\cdots [1,n-1])\leq c([n,0]+[n-1,0]+\cdots [1,n-1])+g(t)\,,
\end{equation*}
which gives \eqref{gro1} for these cases.

Now, we proceed to consider the case $(0,n)$ separately. Here, the estimation of the terms $A_{\beta}$ is carried out as in the former cases of order $n$. However, for $A_{(0,0)}$, instead of \eqref{24} we obtain,
 \begin{align*} |A_{(0,0)}|&\leq c \|\partial_x u\|_{L^\infty}[\alpha_1,\alpha_2]+c\|u\|_{L^\infty}\int(\py^n u)^2\,\chi'\\
 &\leq c[\alpha_1,\alpha_2]+[1,n-1]''\,.
 \end{align*}
  Notice that the case $(1,n-1)$ is of order $n$, but is already a former case. Therefore, taking into account that all cases $(n,0),\cdots (1,n-1)$ are former cases we obtain the inequality 
  \begin{equation*}
  \frac{d}{dt}[0,n]\leq [0,n]+g(t),
  \end{equation*}
  thus giving \eqref{gro1} for this case.
  
   To justify the above formal computations we shall follow the following standard argument.
   
    Consider data $u_0^{\tau}=\rho_{\tau} \ast u_0$ with $\rho\in C^{\infty}_0(\R^2)$, $\text{supp}\,\rho\in B_1(0)=\{z\in \R^2: |z|<1\}$, $\;\rho\ge 0$, $\;\;\displaystyle\int\limits_{\R^2}\rho(z)\,dz=1$ and
 \begin{equation*}
  \rho_{\tau}(z)=\frac{1}{\tau^2}\, \rho\big(\frac{z}{\tau}\big), \;\;\tau>0.
 \end{equation*}

For $\tau>0$ consider  the  solutions $u^{\tau}$ of the IVP \eqref{KP} with data $u_0^{\tau}$  where $(u^{\tau})_{\tau>0}\subseteq C([0,T]: H^{\infty}(\R))$.

 Using the continuous dependence of the solution upon the data we have that
\begin{equation}
\label{123}
\sup_{t\in [0,T]}\;\| u^{\tau}(t)-u(t)\|_{s,2}\,\downarrow 0\;\;\;\;\text{as}\;\;\;\;\tau\,\downarrow 0\text{\hskip10pt for \hskip5pt} s>2.
\end{equation}
 Applying our argument  to the smooth solutions $u^{\tau}(\cdot,t)$ one gets that
\begin{equation}\label{124}
\sup_{[0,T]}\,\int\limits_{\R^2}(\partial^{\alpha} u^{\tau})^2\, \chi_{\epsilon,b}(x+vt)\,dxdy \le c_0
\end{equation}
for any $\epsilon>0$, $b\ge 5\epsilon$, $v>0$, $c_0= c_0(\epsilon;b;v)>0$  but independent of $\tau>0$.

Combining \eqref{123} and \eqref{124} and a weak compactness argument one gets that
\begin{equation}\label{125}
\sup_{[0,T]}\, \int\limits_{\R^2} (\partial^{\alpha} u)^2\, \chi_{\epsilon,b}(x+vt)\,dxdy\le c_0
\end{equation}
which is the desired result. 

This completes the proof of Theorem \ref{th.1}.

  \section*{Acknowledgments} P. I. was supported by Colciencias, Fondo nacional de financiamiento para la ciencia, la
tecnolog\'{i}a y la innovaci\'on Francisco Jos\'e de Caldas, project Ecuaciones diferenciales dispersivas y el\'{i}pticas no lineales, code 111865842951.  
F. L. was partially supported by CNPq and FAPERJ/Brazil. G. P. was  supported by a NSF grant  DMS-1101499.

\end{document}